\documentclass[11pt]{amsart}
\usepackage{epsfig}
\usepackage{amsmath}
\usepackage{amsfonts}
\usepackage{amsthm}
\usepackage{graphics}
\usepackage{amsfonts}
\usepackage{amsmath}
\usepackage{color}
\usepackage{amsthm}
\usepackage{float}
\usepackage{graphicx}

\newcommand{\lcr}{\raisebox{-5pt}{\mbox{}\hspace{1pt}
                  \epsfig{file=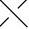}\hspace{1pt}\mbox{}}}
\newcommand{\ift}{\raisebox{-5pt}{\mbox{}\hspace{1pt}
                  \epsfig{file=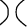}\hspace{1pt}\mbox{}}}
\newcommand{\zer}{\raisebox{-5pt}{\mbox{}\hspace{1pt}
    \epsfig{file=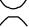}\hspace{1pt}\mbox{}}}

\newtheorem{theorem}{Theorem}[section]
\newtheorem{lemma}[theorem]{Lemma}
\newtheorem{proposition}[theorem]{Proposition}

\theoremstyle{definition}

\theoremstyle{remark}
\newtheorem{remark}[theorem]{Remark}

\numberwithin{equation}{section}

\title[Recursive relation in complement of $(2p+1,2)$ torus knot]{A recursive relation in the complement of the $(2p+1,2)$ torus knot}
\author{Sunday Esebre}
\address{Department of Mathematics and Statistics, 
Texas Tech University, Lubbock, TX 79409}
\email{sunday.esebre@ttu.edu}
\author{R{\u{a}}zvan Gelca}
\address{Department of Mathematics and Statistics, 
Texas Tech University, Lubbock, TX 79409}
\email{rgelca@gmail.com}



\begin{document}
\maketitle

\begin{abstract}
  It is known that the colored Jones polynomials of a knot in the 3-dimensional
  sphere satisfy recursive relations, it is also known that
  these recursive relations come from recurrence polynomials which have
  been related, by the AJ conjecture, to the geometry of the knot complement.
  In this paper we propose a new line of thought, by extending the
  concept of colored Jones polynomials to knots in  the 3-dimensional manifold
  such as  a knot complement, and then examining the case of one particular
  knot in the complement of the $(2p+1,2)$ torus knot for which an analogous
  recursive relation exists, and moreover, this relation
  has an associated recurrence polynomial. Part of our study consists of
  the writing  in the standard basis of the genus two handlebody of
  two families of skeins in this handlebody. 
\end{abstract}

\section{Introduction}

In 1984 V.F.R Jones has discovered a polynomial invariant for
knots in the 3-dimensional sphere \cite{jones}.
E. Witten has explained the Jones polynomial using a quantum field theory
whose action functional is the Chern-Simons functional \cite{witten},
showing that the
Jones polynomial evaluated at a root of unity is the expected value
of the trace of the
holonomy along the knot of an $su(2)$-connection, which holonomy
is computed in the fundamental representation of $SU(2)$. Witten has brought
to  attention the same expected value computed for other possible
representations of  $SU(2)$, among which a special role is played by
the irreducible representations.
And as there is one irreducible representation of $SU(2)$
of dimension $n+1$ for each $n\geq 0$, there is a corresponding knot invariant,
called the $n$th colored Jones polynomial of the knot.
The colored Jones polynomials of knots have been constructed rigorously
in \cite{kirillovreshetikhin} and \cite{reshetikhinturaev} 
using quantum groups.  There exists a slightly modified version of
Witten's theory, based on the Kauffman bracket \cite{kauffman}, which
can be found in \cite{kauffmanlins}, with its own version of
colored Jones polynomials, what we prefer to call the colored Kauffman
brackets.

The combinatorial nature of Witten's Chern-Simons theory is expressed
in skein relations, and these skein relations have led J. Przytycki
to introduce the concept of a skein module in an attempt to capture the
combinatorial aspects of Chern-Simons theory \cite{przytycki}. Our focus is
on the Kauffman bracket skein modules. The Kauffman bracket skein module
of a 3-dimensional oriented manifold $M$ is defined as follows.
Let $\mathcal{L}$ be the set of isotopy classes of framed links in 
the manifold $M$, including the empty link. 
Consider the free ${\mathbb C}[t,t^{-1}]$-module with basis $\mathcal{L}$,
and factor it by the smallest subspace containing all expressions
of the form $\displaystyle{\lcr-t\zer-t^{-1}\ift}$
and 
$\bigcirc+t^2+t^{-2}$, where the links in each expression are
identical except in a ball in which they look like depicted.
The resulting  quotient is the Kauffman bracket skein module
of $M$, denoted by $K_t(M)$. 

It is in the context of Kauffman bracket skein modules that Ch. Frohman has
discovered that the colored Jones polynomials (or rather the colored Kauffman
brackets) of a knot in the 3-dimensional sphere are related to one
another \cite{frgelo}. The relation was expressed as an ``orthogonality''
between the vector with entries the colored Jones polynomials and
a vector computed from a deformed version of the A-polynomial.
The orthogonality relation was further interpreted as a linear recursive
relation for colored Jones polynomials by the second author in
\cite{gelcaproc}. This research has been further refined by S. Garoufalidis
and T.T.Q. Le in \cite{garle} using the concept of $q$-holonomicity, to
show that such recursive relations are build in the very definition
of the colored Jones polynomials. In this context they introduced
the concepts of recurrence polynomials and recurrence ideals for
the colored Jones polynomials of knots. It is important to point out
that the Garoufalidis-Le theory is for the actual colored Jones polynomials
in the way they arise in the Reshetikhin-Turaev theory based on quantum groups,
and not for their slightly
modified Kauffman bracket versions.

In this paper we advance  the problem of finding relations among colored
Jones polynomials to a different setting.
We do this with the example of one knot in the complement of the
$(2p+1,2)$ torus knot. 
Let us recall that it was shown in \cite{bullock} that the Kauffman bracket skein module of the complement in the three dimensional sphere of a regular neighborhood of the  $(2p+1,2)$ torus
knot (in short the complement of the $(2p+1,2)$ torus knot) is a free ${\mathbb C}[t,t^{-1}]$-module with a basis consisting of the
  skeins $$x^my^n,  \quad m\geq 0,\quad  0\leq n\leq p,$$  where $x$ and $y$ are the curves shown in
  Figure~\ref{basistorknot}. Our convention here and throughout the paper is
  that we use the blackboard framing of curves, and that the monomial
  $x^my^n$ means the multicurve consisting of $m$ parallel copies of $x$ and
  $n$ parallel  copies of $y$.
\begin{figure}[h]
  \centering
\scalebox{.4}{\input{basistorknot.pstex_t}}

\caption{}
\label{basistorknot}
\end{figure}

  We recall the two families of (normalized) Chebyshev polynomials, those of
  the first kind, $T_n(\xi)$, defined by $T_n(2\cos \theta)=2\cos n\theta$,
  and those of the second kind, $S_n(\xi)$, defined
  by $S_n(2\cos\theta)=\sin (n+1)\theta/\sin \theta$, where $n$ ranges
  over all integers, positive and negative.
  
  For the Kauffman bracket skein module of the complement of the torus knot
  it is then  sensible to  use  the basis $$S_m(x)S_n(y), \quad m\geq 0, \quad 0\leq n\leq p,$$ 
   motivated by the fact that the polynomial $S_n(\xi)$ is the character of the 
  $n+1$st irreducible representation of the Lie group $SU(2)$, and so $S_m(x)S_n(y)$ would correspond,
  in the Kauffman bracket picture, to the link $x\cup y$ decorated by the $m+1$st and $n+1$st irreducible
  representations of this group (or of its quantum version).   

  The present paper is focussed on just one example, and this example is the
  knot in the complement of the $(2p+1,2)$ torus knot that  is the curve $y$
  (endowed with the blackboard framing).  
Our main result is the following:

\begin{theorem}\label{mainthm}
  For all $n\in {\mathbb Z}$, the following identity holds in
  the Kauffman bracket skein module of the complement of the $(2p+1,2)$ torus
  knot
  \begin{eqnarray*}
t^{-2n-1}S_{p+n}(y)+t^{2n+1}S_{p-n-1}(y)=(-1)^nS_{2n}(x)(tS_{p-1}(y)+t^{-1}S_p(y)). 
    \end{eqnarray*}
    \end{theorem}

This formula was noticed by J. Sain and
proved for $0\leq n\leq p+1$ in \cite{gelcasain}, but unfortunately that proof does not extend to other $n$. In this paper we
use some results of the second author and his collaborators
about the Kauffman bracket skein module of the genus two handlebody
to find a different proof that works
in general. In the process we address a problem raised
by R.P. Bakshi and J. Przytycki in conjuction with their work
on connected sums of handlebodies from \cite{rheajozef},
the problem being about expressing a
certain skein in
the genus two handlebody in terms of the standard basis elements. 

By using Theorem~\ref{mainthm} in \S\ref{recforjones} of this paper we put the
colored Jones polynomials of  $y$  in the context of theory developed by
Garoufalidis and Le in \cite{garle}. 
The question, of course, is what should  the analogue of the $n$th colored
Jones polynomial of a knot $K$ in
an arbitrary manifold be? Based on the considerations explained in \cite{annrazvan},
we define these ``colored Jones polynomials'' to be the skeins  $S_n(K)$, $n\geq 0$, inside the skein module defined by the skein
relations derived in
\cite{kirbymelvin} for the Reshetikhin-Turaev theory \cite{reshetikhinturaev}
(see \S\ref{recforjones} below for the definition).
If we apply mutatis mutandis the method of \cite{garle},  then we find
the  recurrence polynomial for the colored Jones polynomials of
the curve $y$ in the complement of the $(2p+1,2)$ torus knot to be
\begin{eqnarray*}
  [L^2-t^4(x^2-2)L+t^8]L^{2p+1}+t^{4p+8}[L^2-(x^2-2)L+1]M^2.
\end{eqnarray*}
This simple polynomial contains all the necessary information for
computing $S_n(y)$ as a linear combination with coefficients
in ${\mathbb C}[t,t^{-1},x]$ of $S_0(y)$, $S_1(y)$, $\ldots$, $S_p(y)$ for $n>p$. Does there
exist a geometric interpretation of this polynomial analogous to
the one found in  \cite{frgelo}, \cite{garle}  for the recurrence polynomials of colored Jones polynomials
of knots in the 3-dimensional sphere?

\section{Some sequences of skeins in the complement of the genus 2 handlebody}

To enhance our ability to prove Theorem~\ref{mainthm}, we have to take a detour
through the skein theory of the genus two handlebody. 
Specifically, we  discuss three sequences of skeins in the
Kauffman bracket skein module of  the genus two handlebody, the third of
which has appeared in our previous work \cite{gelcanagasato}.
Our discussion requires some basic knowledge about linear recursive sequences;
a good reference for the necessary techniques is \cite{andrica}.

We interpret the genus two handlebody as the cylinder over the twice
punctured disk and we represent it schematically  sideways,
by drawing only the two curves that
trace the punctures in the cylinder. Przytycki \cite{przytycki} has
shown that the Kauffman bracket skein module of the genus two handelbody
is free with basis $x^my^nz^k$, $m,n,k\geq 0$,
where $x$ and $z$ are curves that are parallel to the boundaries of
the two open disks that  have been removed, and $y$ is a curve parallel
to the boundary of the original disk. The curves  $x,y,z$  are  shown
in Figure~\ref{xyz}.

\begin{figure}[h]
  \centering
\scalebox{.5}{\input{xyz.pstex_t}}

\caption{}
\label{xyz}
\end{figure}

The first two sequences of skeins that we have in mind have been
introduced in \cite{rghw}. They are 
$X_1*y^n$ and $Y_1*y^n$ depicted in Figure~\ref{XYyn}.
\begin{figure}[h]
  \centering
\scalebox{.5}{\input{XYyn.pstex_t}}

\caption{}
\label{XYyn}
\end{figure}
Modifying appropriately  the argument of  Lemma~2.1 in \cite{rghw}
we obtain

\begin{lemma}\label{recforx1y1}
  The sequences $X_1*y^n$, $Y_1*y^n$ satisfy the recursive relations
  \begin{eqnarray*}
&& X_1*y^{n+1}=t^4yX_1*y^n+(t^{-2}-t^6)Y_1*y^n+(1-t^4)(x^2+z^2)y^n\\
    &&Y_1*y^{n+1}=t^{-4}yY_1*y^n+(t^2-t^{-6})X_1*y^n+2(1-t^{-4})xzy^n,\\
    && X_1*y^0=-t^4y-t^2xz, \quad Y_1*y^0=-t^2-t^{-2}.
    \end{eqnarray*}
\end{lemma}




We are interested in finding explicit formulas for $X_1*y^n$ and $Y_1*y^n$,
but as experience has taught us, it is better to replace the ``powers'' of
$y$ by Chebyshev polynomials in $y$. So  instead we will find explicit formulas
for $X_1*T_n(y)$ and $Y_1*T_n(y)$, where $T_n$ is the Chebyshev polynomial
of first kind defined in the introduction.

\begin{theorem}\label{formforXY}
  The following formulas hold
  \begin{eqnarray*}
    &&    X_1*T_n(y)= -t^{4n+4}S_{n+1}(y)-t^{-4n}S_{n-1}(y)+t^{4n}S_{n-1}(y)\\&&\quad
    +t^{-4n+4}S_{n-3}(y) -t^{4n+2}xzS_n(y)+t^{-4n+2}xzS_{n-2}(y)
    \\&&\quad +(1-t^{4n})\sum_{k=0}^{n-1}t^{-4k}(x^2+z^2)S_{n-2k-1}(y)\\&&\quad +2(1-t^{4n})\sum_{k=0}^{n-1}t^{-4k-2}xzS_{n-2k-2}(y),\quad n\geq 1,\\
    &&Y_1*T_n(y)=-(t^{4n+2}+t^{-4n-2})S_n(y)+(t^{-4n}-t^{4n})xzS_{n-1}(y)\\
    &&\quad +(t^{4n-2}+t^{-4n+2})S_{n-2}(y)+   (1-t^{4n})(x^2+z^2)\sum_{k=0}^{n-1}t^{-4k-2}S_{n-2k-2}(y)\\
    &&\quad +2(1-t^{4n})xz\sum_{k=0}^{n-1}t^{-4k-4}S_{n-2k-3}(y),\quad n\geq 1.
    \end{eqnarray*}
\end{theorem}

\begin{proof} The argument is based on the recursive relation for the vector
  $$(X_1*y^n,Y_1*y^n)$$ exhibited in
  Lemma~\ref{recforx1y1}.
We introduce the auxiliary variable $w$ so that $y=w+w^{-1}$
($w$ has no geometric meaning, it is used solely for computations).
The coefficient matrix of the recursive relation,
\begin{eqnarray*}
  A=\left(
  \begin{array}{cc}
    t^4y&t^{-2}-t^6\\
    t^2-t^{-6}&t^{-4}y
  \end{array}
  \right),
  \end{eqnarray*}
has eigenvalues
\begin{eqnarray*}
\lambda_{1,2}=\frac{1}{2}[(t^4+t^{-4})y\pm (t^4-t^{-4})\sqrt{y^2-4}]=t^4w^{\pm 1}+t^{-4}w^{\mp 1},
  \end{eqnarray*}
with eigenvectors $(t^2w,1)$ and $(w^{-1},t^{-2})$, respectively.
Hence this coefficient matrix is diagonalized as
\begin{eqnarray*}
  \frac{1}{w-w^{-1}} \left(
    \begin{array}{cc}
      t^2w& w^{-1}\\
      1 &t^{-2}
    \end{array}
    \right)\left(
    \begin{array}{cc}
      t^4w+t^{-4}w^{-1} &0\\
      0 &t^4w^{-1}+t^{-4}w
    \end{array}
    \right)\left(
    \begin{array}{cc}
      t^{-2} &-w^{-1}\\
      -1&t^2w
    \end{array}
    \right)
  \end{eqnarray*}
Now we split the sequence $(X_1*y^n, Y_1*y^n)$ into $(a_n,b_n)+(c_n,d_n)$
where $(a_n,b_n)$ satisfies the homogeneous recursive relation
\begin{eqnarray*}
  \left(
  \begin{array}{c}
    a_{n+1}\\
    b_{n+1}
  \end{array}
  \right)=A \left(
  \begin{array}{c}
    a_n\\
    b_n
  \end{array}
  \right),\quad a_0=X_1*y^0,\quad b_0=Y_1*y^0,
\end{eqnarray*}
and $(c_n,d_n)$ satisfies the nonhomogenous recursive relation with trivial initial condition
\begin{eqnarray*}
  \left(
  \begin{array}{c}
    c_{n+1}\\
    d_{n+1}
  \end{array}
  \right)=A \left(
  \begin{array}{c}
    c_n\\
    d_n
  \end{array}
  \right)+
  \left(\begin{array}{c}
    (1-t^4)(x^2+z^2)y^n\\
    2(1-t^{-4})xzy^n
  \end{array}\right)
    ,\quad c_0=d_0=0.
\end{eqnarray*}
We obtain
\begin{eqnarray*}
    \left(\begin{array}{c}a_n\\b_n\end{array}\right)=
   \frac{1}{w-w^{-1}} \left(
    \begin{array}{cc}
      t^2w& w^{-1}\\
      1 &t^{-2}
    \end{array}
    \right)\left(
    \begin{array}{cc}
      (t^4w+t^{-4}w^{-1})^n &0\\
      0 &(t^4w^{-1}+t^{-4}w)^n
    \end{array}
    \right)\\ \times \left(
    \begin{array}{cc}
      t^{-2} &-w^{-1}\\
      -1&t^2w
    \end{array}
    \right)\left(
    \begin{array}{c}
      -t^4(w+w^{-1})-t^2xz\\
      -t^2-t^{-2}
    \end{array}
    \right).
  \end{eqnarray*}
  Using the equalities $$T_n(tw+t^{-1}w^{-1})=t^nw^n+t^{-n}w^{-n},T_n(t^{-1}w+tw^{-1})=t^{-n}w^n+t^{n}w^{-n}$$ we deduce that  the ``homogeneous'' part of $(X_1*T_n(y),Y_1*T_n(y))$ is
   \begin{eqnarray*}
  \frac{1}{w-w^{-1}} \left(
    \begin{array}{cc}
      t^2w& w^{-1}\\
      1 &t^{-2}
    \end{array}
    \right)\left(
    \begin{array}{cc}
      t^{4n}w^n+t^{-4n}w^{-n} &0\\
      0 &t^{4n}w^{-n}+t^{-4n}w^n
    \end{array}
    \right)\\
    \times\left(
    \begin{array}{cc}
      t^{-2} &-w^{-1}\\
      -1&t^2w
    \end{array}
    \right)\left(
    \begin{array}{c}
      -t^4(w+w^{-1})-t^2xz\\
      -t^2-t^{-2}
    \end{array}
    \right)   
   \end{eqnarray*}
   Using the fact that $$S_n(y)=\frac{w^{n+1}-w^{-n-1}}{w-w^{-1}}$$ we obtain
   that this is further equal to
   \begin{eqnarray*}
     \left(
     \begin{array}{cc}
       t^{4n}S_n(y)-t^{-4n}S_{n-2}(y)& -t^2(t^{4n}-t^{-4n})S_{n-1}(y)\\
         -t^{-2}(t^{-4n}-t^{4n})S_{n-1}(y)&t^{-4n}S_n(y)-t^{4n}S_{n-2}(y)
       \end{array}
       \right)\left(
    \begin{array}{c}
      -t^4y-t^2xz\\
      -t^2-t^{-2}
    \end{array}
    \right)   
   \end{eqnarray*}
   So the ``homogeneous'' part of $X_1*T_n(y)$ is
      \begin{eqnarray*}
        -t^{4n+4}S_{n+1}(y)-t^{-4n}S_{n-1}(y)+t^{4n}S_{n-1}(y)+t^{-4n+4}S_{n-3}(y)\\
        -t^{4n+2}xzS_n(y)+t^{-4n+2}xzS_{n-2}(y),
      \end{eqnarray*}
      while the ``homogeneous'' part of $Y_1*T_n(y)$ is 
      \begin{eqnarray*}
-(t^{4n+2}+t^{-4n-2})S_n(y)+(t^{-4n}-t^{4n})xzS_{n-1}(y)\\+(t^{4n-2}+t^{-4n+2})S_{n-2}(y).
      \end{eqnarray*}

      On the other hand
      \begin{eqnarray*}
        &&\left(
        \begin{array}{c}
          c_n\\
          d_n
        \end{array}
        \right)=\frac{(t^2-t^{-2})}{w-w^{-1}} \left(
    \begin{array}{cc}
      t^2w& w^{-1}\\
      1 &t^{-2}
    \end{array}
    \right)\\
    &&\times \left(
    \begin{array}{cc}
      \frac{\displaystyle{t^{4n}w^n+t^{-4n}w^{-n}-w^n-w^{-n}}}{\displaystyle{t^4w+t^{-4}w^{-1}-w-w^{-1}}} &0\\
      0 &\frac{\displaystyle{t^{4n}w^{-n}+t^{-4n}w^n-w^n-w^{-n}}}{\displaystyle{t^4w^{-1}+t^{-4}w-w-w^{-1}}}
    \end{array}
    \right)\\
    &&\times\left(
    \begin{array}{cc}
      t^{-2} &-w^{-1}\\
      -1&t^2w
    \end{array}
    \right)\left(
    \begin{array}{c}
     - t^2(x^2+z^2)\\
      2t^{-2}xz
    \end{array}
    \right).
               \end{eqnarray*}
      Rewrite this expression as
      \begin{eqnarray*}
&&\frac{t^{-2}(t^{4n}-1)}{w-w^{-1}} \left(
    \begin{array}{cc}
      t^2w& w^{-1}\\
      1 &t^{-2}
    \end{array}
    \right)\left(
    \begin{array}{cc}
      \frac{\displaystyle{w^n-t^{-4n}w^{-n}}}{\displaystyle{w-t^{-4}w^{-1}}}&0\\
      0&\frac{\displaystyle{w^{-n}-t^{-4n}w^{n}}}{\displaystyle{w^{-1}-t^{-4}w}}
    \end{array}
    \right)\\
  &&  \times\left(
    \begin{array}{cc}
      t^{-2} &-w^{-1}\\
      -1&t^2w
    \end{array}
    \right)\left(
    \begin{array}{c}
     - t^2(x^2+z^2)\\
      2t^{-2}xz
    \end{array}
    \right)\\
    &&=\frac{t^{-2}(t^{4n}-1)}{w-w^{-1}}
   \left( \begin{array}{cc}
      A&B\\C&D
    \end{array}
    \right)
    \left(
    \begin{array}{c}
     - t^2(x^2+z^2)\\
      2t^{-2}xz
    \end{array}
    \right),
        \end{eqnarray*}
      where
      \begin{eqnarray*}
&& A=\frac{w^{n+1}-t^{-4n}w^{-n+1}}{w-t^{-4}w^{-1}}-\frac{w^{-n-1}-t^{-4n}w^{n-1}}{w^{-1}-t^{-4}w}\\
        &&   B=-t^2\left(\frac{w^n-t^{-4n}w^{-n}}{w-t^{-4}w^{-1}}-\frac{w^{-n}-t^{-4n}w^n}{w^{-1}-t^{-4}w}\right),\\
        &&C=t^{-2}\left(\frac{w^n-t^{-4n}w^{-n}}{w-t^{-4}w^{-1}}-\frac{w^{-n}-t^{-4n}w^n}{w^{-1}-t^{-4}w}\right),\\
        &&D=-\frac{w^{n-1}-t^{-4n}w^{-n-1}}{w-t^{-4}w^{-1}}+\frac{w^{-n+1}-t^{-4n}w^{n+1}}{w^{-1}-t^{-4}w}.
        \end{eqnarray*}
      And we have
      \begin{eqnarray*}
        &&\frac{A}{w-w^{-1}}=\sum_{k=0}^{n-1}t^{-4k}S_{n-2k-1}(y),\quad \frac{B}{w-w^{-1}}=-\sum_{k=0}^{n-1}t^{-4k+2}S_{n-2k-2}(y),\\
        &&\frac{C}{w-w^{-1}}=\sum_{k=0}^{n-1}t^{-4k-2}S_{n-2k-2}(y),\quad \frac{D}{w-w^{-1}}=-\sum_{k=0}^{n-1}t^{-4k}S_{n-2k-3}(y).
      \end{eqnarray*}
      After a final multiplication we obtain that the nonhomogeneos parts of
      $X_1*T_n(y)$ and $Y_1*T_n(y)$ are, respectively,
      \begin{eqnarray*}
        (1-t^{4n})\sum_{k=0}^{n-1}t^{-4k}(x^2+z^2)S_{n-2k-1}(y)+2(1-t^{4n})\sum_{k=0}^{n-1}t^{-4k-2}xzS_{n-2k-2}(y),\\
        (1-t^{4n})(x^2+y^2)\sum_{k=0}^{n-1}t^{-4k-2}S_{n-2k-2}(y)+2(1-t^{4n})xz\sum_{k=0}^{n-1}t^{-4k-4}S_{n-2k-3}(y).
        \end{eqnarray*}
      The conclusion follows.
\end{proof}

As hinted in the introduction, in the work of Bakshi and  Przytycki has appeared  the question of expressing the skein from Figure~\ref{sigman} in terms of the basis $x^my^nz^k$ of the Kauffman bracket skein module of the genus two handlebody. Again, it is advantageous to work with curves decorated by
Chebyshev polynomials instead of ``powers'', we will therefore compute instead
the skein in which $y^n$ is replaced by $T_n(y)$, and let us call this skein
$\sigma_n$. And we use the basis $S_m(x)S_n(y)S_k(z)$, $m,n,k\geq 0$.

\begin{figure}[h]
  \centering
\scalebox{.6}{\input{sigman.pstex_t}}

\caption{}
\label{sigman}
\end{figure}

\begin{proposition}
  We have
  \begin{eqnarray*}
&&\sigma_n=-t^{4n+5}S_{n+1}(y)-t^{-4n+1}S_{n-1}(y)+t^{4n+1}S_{n-1}(y)+t^{-4n+5}S_{n-3}(y)\\
    &&  \quad -t^{4n+3}S_1(x)S_n(y)S_1(z)+t^{-4n+3}S_1(x)S_{n-2}(y)S_1(z)\\
    &&\quad+t^{-1}S_1(x)S_n(y)S_1(z)+t^{-1}S_1(x)S_{n-2}(y)S_1(z)
    \\&&\quad+t(1-t^{4n})\sum_{k=0}^{n-1}t^{-4k}(S_2(x)+S_2(z)+2)S_{n-2k-1}(y) \\
    &&\quad+2t^{-1}(1-t^{4n})\sum_{k=0}^{n-1}t^{-4k}S_1(x)S_{n-2k-2}(y)S_1(z).
    \end{eqnarray*}
  \end{proposition}

\begin{proof}
  This follows from the fact that $$\sigma_n=tX_1*T_n(y)+t^{-1}xzT_n(y)=tX_1*T_n(y)+t^{-1}xz(S_n(y)+S_{n-2}(y))$$ and then  applying
  Theorem~\ref{formforXY}. 
  \end{proof}

Finally, let us recall the skeins $X_i$ from the Kauffman
bracket skein module of the genus two handlebody, which were defined 
in \cite{gelcanagasato} and are depicted in Figure~\ref{xicurves}.

\begin{figure}[h]
  \centering
\scalebox{.3}{\input{xicurves.pstex_t}}

\caption{}
\label{xicurves}
\end{figure}

Adapting the proof of Lemma~1 from \cite{gelcanagasato} to the case where $x\neq z$, we see that the $X_i$  satisfy the recursive relation
\begin{eqnarray*}
X_{i+2}=t^2yX_{i+1}-t^4X_i-2t^2xz,\quad X_0=-t^2-t^{-2}, X_1=-t^4y-t^2xz.
  \end{eqnarray*}
and  consequently
 \begin{eqnarray*}
   X_{i}=-t^{-2i-2}S_{i}(y)-t^{-2i}xzS_{i-1}(y)+t^{-2i+2}S_{i-2}(y)\\
   -2t^{-2i+2}xz\sum_{k=0}^{i-2}t^{2k}S_{i-k-2}(y).
       \end{eqnarray*}
 This formula has been proved in \cite{gelcanagasato}; it  can be easily
 verified by induction,
 and it can be determined using  the fact that $t^{-2i}S_{i}(y)$
 and $t^{-2i-2}S_{i-1}$ form a basis for the space of solutions to
 the homogeneous recursion  $X_{i+2}=t^2yX_{i+1}-t^4X_i$.

\section{Proof of the main result}

In this section we prove  Theorem~\ref{mainthm}. Because $S_{-n}(y)=-S_{n-2}(y)$, we
only need to check the identity
  \begin{eqnarray*}
t^{-2n-1}S_{p+n}(y)+t^{2n+1}S_{p-n-1}(y)=(-1)^nS_{2n}(x)(tS_{p-1}(y)+t^{-1}S_p(y)) 
  \end{eqnarray*}
  for $n\geq 0$. This identity has been  checked already in \cite{gelcasain} for $n=0,1,2,\ldots, p+1$, but,
  as mentioned before,  that proof cannot be extended for larger $n$. 

We will use instead  the equality derived in Figure~\ref{handleslide}.
\begin{figure}[h]
  \centering
\scalebox{.3}{\input{handleslide.pstex_t}}

\caption{}
\label{handleslide}
\end{figure}
This equality translates to
\begin{eqnarray*}
\overline{X_1*y^n}=y^n\overline{X_{2p}}, 
\end{eqnarray*}
where $X_1*T_n(y)$ and $X_{2p}$ are the skeins discussed in the previous
section (viewed as lying inside the handlebody marked in the figure with dots),
$\overline{\Sigma}$ denotes the mirror image of the skein $\Sigma$
(with respect to a natural reflection of the handlebody onto itself) and on
the right-side we have the  product of $y^n$ and $\overline{X_{2p}}$ defined
by the cylinder structure of the genus two handlebody. Note that when passing
to the mirror image $t$ should be replaced by $t^{-1}$.

As a consequence of this identity we obtain
\begin{eqnarray}\label{handleslideform}
\overline{X_1*T_n(y)}=T_n(y)\overline{X_{2p}}, 
\end{eqnarray}
which is the key ingredient in the proof of Theorem~\ref{mainthm}.

Using  the (trigonometric) identity
 $S_k(x)T_n(x)=S_{k+n}(x)+S_{k-n}(x)$ and Theorem~\ref{formforXY}
 we can write (\ref{handleslideform}) explicitly as
\begin{eqnarray*}
			&&-t^{-4n-4}S_{n+1}(y)-t^{4n}S_{n-1}(y)+t^{-4n}S_{n-1}(y)+t^{4n-4}S_{n-3}(y)\\&&+   t^{-4n-2}x^2S_n(y)-2t^{-2}x^2S_n(y)-t^{4n-2}x^2S_{n-2}(y)+2t^{-2}x^2S_{n-2}(y)\\&&+2t^{-2}(1-t^{-4n})x^2\sum_{k=0}^{2n-1}t^{2k}S_{n-k}(y)+t^{-4p-2}S_{2p+n}(y)+t^{-4p-2}S_{2p-n}(y)\\&&+t^{-4p}x^2S_{2p-1+n}(y)
  +t^{-4p}x^2S_{2p-1-n}(y)-t^{-4p+2}S_{2p-2+n}(y)\\&&-t^{-4p+2}S_{2p-2-n}(y)+2t^{-4p+2}x^2\sum_{k=0}^{2p-2}t^{2k}[S_{2p-k-2+n}(y)+S_{2p-k-2-n}(y)]\\
  &&=0.
\end{eqnarray*}
Both this recursive relation
and the one from the main theorem (Theorem~\ref{mainthm}) completely determine the
values of $S_{n}(y)$ for $n>p$ from $S_0(y)$, $S_1(y)$, $\ldots$, $S_p(y)$.
To complete the proof of Theorem\ref{mainthm} it suffices to show that
the sequence $S_n(y)$ defined by the recursive relation from the statement
of this theorem also satisfies the above identity. So from this moment on
we assume that $S_n(y)$ is the sequence defined by the recursive relation
from the statement of Theorem~\ref{mainthm} and we prove that it satisfies this
identity.

For the proof, we further transform this desired identity into
			\begin{eqnarray*}
			&&t^{-2p+2n-1}(t^{-2p-2n-1}S_{2p+n}(y)-t^{2p+2n+1}S_{n-1}(y))\\&&+t^{-2p-2n-1}(t^{-2p+2n-1}S_{2p-n}(y)+t^{2p-2n+1}S_{n-1}(y))\\
			  &&+t^{-2p+2n-1}x^2(t^{-2p-2n+1}S_{2p-1+n}(y)-t^{2p+2n-1}S_{n-2}(y))\\&&+t^{-2p-2n-1}x^2(t^{-2p+2n+1}S_{2p-n-1}(y)+t^{2p-2n-1}S_n(y))\\&&-t^{-2p+2n-1}(t^{-2p-2n+3}S_{2p-2+n}(y)-t^{2p+2n-3}S_{n-3}(y))\\
                          &&-t^{-2p-2n-1}(t^{-2p+2n+3}S_{2p-2-n}(y)+t^{2p-2n-3}S_{n+1}(y))\\
		          &&	=2t^{-2}x^2\left[S_n(y)+S_{n-2}(y)+(t^{-4n}-1)\sum_{k=0}^{2n-1}t^{2k}S_{n-k}(y)\right.\\
                            &&\left.-\sum_{k=0}^{2p-2}t^{-4p+2k+4}(S_{2p-k-2+n}(y)+S_{2p-k-2-n}(y))\right],
			\end{eqnarray*}
			and then rewrite it as
			\begin{eqnarray*}
			  &&t^{-2p+2n-1}(t^{-2p-2n-1}S_{2p+n}(y)-t^{2p+2n+1}S_{n-1}(y))\\
                          &&+t^{-2p-2n-1}(t^{-2p+2n-1}S_{2p-n}(y)+t^{2p-2n+1}S_{n-1}(y))\\
			  &&+t^{-2p+2n-1}x^2(t^{-2p-2n+1}S_{2p-1+n}(y)
                          -t^{2p+2n-1}S_{n-2}(y))\end{eqnarray*}
                        \begin{eqnarray*}&&+t^{-2p-2n-1}x^2(t^{-2p+2n+1}S_{2p-n-1}(y)+t^{2p-2n-1}S_n(y))\\&&-t^{-2p+2n-1}(t^{-2p-2n+3}S_{2p-2+n}(y)
                          -t^{2p+2n-3}S_{n-3}(y))\\&&-t^{2p-2n-1}(t^{-2p+2n+3}S_{2p-2-n}(y)+t^{2p-2n-3}S_{n+1}(y))\\
		          &&	=2t^{-2}x^2\left[S_n(y)-S_{n-2}(y)+(t^{-4n}-1)\sum_{k=0}^{2n-1}t^{2k}S_{n-k}(y)\right.\\
                            &&\left.-\sum_{k=0}^{2p-2}t^{-2k}(S_{k+n}(y)+S_{k-n}(y))\right].
			\end{eqnarray*}
			Finally, we bring it into the form
			\begin{eqnarray*}
                          &&t^{-2p+2n-1}(t^{-2p-2n-1}S_{2p+n}(y)-t^{2p+2n+1}S_{n-1}(y))\\
                          &&+t^{-2p-2n-1}(t^{-2p+2n-1}S_{2p-n}(y)+t^{2p-2n+1}S_{n-1}(y))\\
&&+t^{-2p+2n-1}x^2(t^{-2p-2n+1}S_{2p-1+n}(y)-t^{2p+2n-1}S_{n-2}(y))\\
                          &&+t^{-2p-2n-1}x^2(t^{-2p+2n+1}S_{2p-n-1}(y)+t^{2p-2n-1}S_n(y))\\
                          &&-t^{-2p+2n-1}(t^{-2p-2n+3}S_{2p-2+n}(y)-t^{2p+2n-3}S_{n-3}(y))\\
 &&-t^{-2p-2n-1}(t^{-2p+2n+3}S_{2p-2-n}(y)+t^{2p-2n-3}S_{n+1}(y))\\
&&=2t^{-2}x^2\left[(t^{-4n}-1)\sum_{k=0}^{2n-1}t^{2k}S_{n-k}(y)-\sum_{k=1}^{2p-2}t^{-2k}(S_{k+n}(y)+S_{k-n}(y))\right].
		        \end{eqnarray*}
                        We will therefore prove that the sequence $S_n(y)$ defined by the
                        recursive relation from the statement of Theorem~\ref{mainthm} satisfies this identity.
                        Using the formula from the statement of the theorem  we deduce that
                        the left-hand side of the identity to be checked is equal to
\begin{eqnarray*}
  &&t^{-2p+2n-1}(-1)^{p+n}S_{2p+2n}(x)Y+t^{-2p-2n-1}(-1)^{p-n}S_{2p-2n}(x)Y\\
  &&+t^{-2p+2n-1}x^2(-1)^{p+n-1}S_{2p+2n-2}(x)Y\\
  &&+t^{-2p-2n-1}x^2(-1)^{p-n-1}S_{2p-2n-2}(x)Y\\
&&-t^{-2p+2n-1}(-1)^{p+n-2}S_{2p+2n-4}(x)Y-t^{-2p-2n-1}(-1)^{p-n-2}S_{2p-2n-4}(x)Y,
\end{eqnarray*}
where we use the short hand notation $Y=t^{-1}S_p(y)+tS_{p-1}(y)$.
  This expression is
 \begin{eqnarray*}
   (-1)^{p+n+1}t^{-1}\left[t^{-2p+2n}(S_{2p+2n-2}(x)+S_{2p+2n-4}(x))Y\right.\\
     \left.+t^{-2p-2n}(S_{2p-2n-2}(x)+S_{2p-2n-4}(x))Y\right].
 \end{eqnarray*}
 Now let us work on the right side of the identity. We have
 
  \begin{eqnarray*}
    &&\sum_{k=1}^{2p-2}t^{-2k}(S_{k+n}(y)+S_{k-n}(y))=\sum_{k=1}^{2p-2}t^{-2k}S_{k+n}(y)\\
 &&   +\sum_{k=1}^{2p-2}t^{-2(2p-1-k)}S_{2p-2-k-n+1}(y)
    =(1-t^{-4n})\sum_{k=1}^{2p-2}t^{-2k}S_{k+n}(y)\\
    &&+\sum_{k=1}^{2p-2}t^{-2p+1+2n}[t^{-2p+2k-2n+1}S_{2p-2-k-n+1}(y)
      +
                  t^{2p-2k+2n-1}S_{k+n}(y)].
                \end{eqnarray*}
  By using the formula from the statement of the theorem,
  this becomes
  \begin{eqnarray*}
\sum_{k=1}^{2p-2}(-1)^{p-n-k-1}t^{-2p+1-2n}S_{2(p-n-k-1)}(x)Y+(1-t^{-4n})\sum_{k=1}^{2p-2}t^{-2k}S_{k+n}(y).
    \end{eqnarray*}
  We are left  to checking that the sequence $S_n(y)$ satisfies the simpler identity
   \begin{eqnarray*}
&&[t^{4n}(S_{2p+2n-2}(x)+S_{2p+2n-4}(x))+S_{2p-2n-2}(x)+S_{2p-2n-4}(x)\\
 &&+\sum_{k=1}^{2p-2}(-1)^{k}x^2S_{2p-2n-2k-2}(x)]Y\\&&
 =(-1)^{p+n}t^{2p-1}x^2(t^{2n}-t^{-2n})\left[\sum_{k=0}^{2n-1}t^{2k}S_{n-k}(y)+\sum_{k=1}^{2p-2}t^{-2k}S_{n+k}(y)\right],                       
   \end{eqnarray*}
   Transform this into
      \begin{eqnarray*}
&&			[t^{4n}(S_{2p+2n-2}(x)+S_{2p+2n-4}(x))
			+S_{2p-2n-2}(x)+S_{2p-2n-4}(x)
\\&&-(-1)^p\sum_{j=1-p}^{p-2}(-1)^{j}x^2S_{2j-2n}(x)]Y\\&&
 =(-1)^{p+n}t^{2p-1}x^2(t^{2n}-t^{-2n})\left[\sum_{k=0}^{2n-1}t^{2k}S_{n-k}(y)+\sum_{k=1}^{2p-2}t^{-2k}S_{n+k}(y)\right],                      
                \end{eqnarray*}
      then into
  \begin{eqnarray*}
&&[t^{4n}(S_{2p+2n-2}(x)+S_{2p+2n-4}(x))+S_{2p-2n-2}(x)+S_{2p-2n-4}(x)\\&&-S_{2p+2n-2}(x)
-S_{2p+2n-4}(x)-S_{2p-2n-2}(x)-S_{2p-2n-4}(x)]Y\\
&&=(-1)^{p+n}t^{2p-1}x^2(t^{2n}-t^{-2n})\left[\sum_{k=0}^{2n-1}t^{2k}S_{n-k}(y)+\sum_{k=1}^{2p-2}t^{-2k}S_{n+k}(y)\right].
  \end{eqnarray*}
  This can be simplified as
  \begin{eqnarray*}
    &&(S_{2p+2n-2}(x)+S_{2p+2n-4}(x))Y\\
     &&=(-1)^{p+n}t^{2p-2n-1}x^2\left[\sum_{k=0}^{2n-1}t^{2k}S_{n-k}(y)+\sum_{k=1}^{2p-2}t^{-2k}S_{n+k}(y)\right].                       
                \end{eqnarray*}
  And now we check that the sequence $S_n(y)$
  defined by the recursive relation from   the statement of
  Theorem~\ref{mainthm} satisfies this identity by induction on $n$. 
  
  Set
  \begin{eqnarray*}
A_n=\sum_{k=1}^{2n-1}t^{2k}S_{n-k}(y)+\sum_{k=1}^{2p-2}t^{-2k}S_{n+k}(y). 
  \end{eqnarray*}
  Then
  \begin{eqnarray*}
    &&A_{n+1}-t^2A_n=t^{-4p+4}S_{n+2p-1}(y)+t^{4p+2}S_{-n}(y)\\
    &&=t^{-4p+4}S_{n+2p-1}(y)-t^{4p+2}S_{n-2}(y)\\
    &&    =t^{-2p+2n+3}(t^{-2p-2n+1}S_{n+2p-1}(y)-t^{2p+2n-1}S_{n-2}(y))\\
    &&=(-1)^{p+n-1}t^{-2p+2n+2}S_{2n+2p-2}(x)Y, 
  \end{eqnarray*}
  where for the last equality we used the identity from the statement of the theorem.
  We are attempting to prove that
  \begin{eqnarray*}
[S_{2n+2p-2}(x)+S_{2n+2p-4}(x)]Y=(-1)^{p+n}t^{2p-2n}x^2A_n
    \end{eqnarray*}
  holds for all $n$ and as  induction step we assume that this equality holds
  for $n$ and prove it for $n+1$. But if we add the equality for $n$ to the one
  for $n+1$ we obtain something obvious, as the following computation shows:
  \begin{eqnarray*}
    &&[S_{2n+2p}(x)+2S_{2n+2p-2}(x)+S_{2n+2p-4}(x)]Y=x^2S_{2n+2p-2}(x)Y\\
    &&=(-1)^{p+n+1}t^{2p-2n-2}x^2(-1)^{p+n-1}t^{-2p+2n+2}S_{2n+2p-2}(x)Y\\
    &&=
    (-1)^{p+n}t^{2p-2n-2}x^2[-A_{n+1}+t^2A_n].
  \end{eqnarray*}
  The induction  is now complete because the base case was checked in
  \cite{gelcasain}.

\section{A recurrence polynomial for the colored Jones polynomials
  of  $y$}\label{recforjones}

In this section we will explain how the recurrence polynomial for the
colored Jones polynomials of the curve $y$ is constructed.
The $n$th  colored Jones polynomial
of a knot $K$ in $S^3$ has been introduced in \cite{witten} as the quantized Wilson line
associated to the knot $K$ and the $n+1$-dimensional irreducible representation
of  $SU(2)$ and has been constructed
rigorously in \cite{kirillovreshetikhin} and \cite{reshetikhinturaev} using the $n+1$st irreducible representation of the
quantum group of $SU(2)$. It is known that the $n$th colored Jones polynomial of the knot
$K$ is
\begin{eqnarray}\label{jokau}
J(K,n)=  (-1)^n\left<S_n(K)\right>
\end{eqnarray}
where $\left<\cdot \right>$ is the Kauffman bracket. The equality (\ref{jokau}) can be placed 
in the setting of skein modules as follows. 

Alongside with the Kauffman bracket skein module $K_t(M)$ of a manifold $M$ one can introduce the skein module of the Reshetikhin-Turaev theory, which was denoted
by $RT_t(M)$ and was defined in \cite{geluri} as
follows. As in the case of the Kauffman bracket skein module we consider the
free ${\mathbb C}[t,t^{-1}]$-module with bases the isotopy classes of
framed links in $M$. To obtain $RT_t(M)$ we factor this module by
the submodule generated by three families of elements.

\begin{figure}[h]
  \centering
   \scalebox{.4}{\input{crossing.pstex_t}}

\caption{}
\label{crossing}
\end{figure}
The first family consists of elements of the form
\begin{eqnarray*}
L\cup O-(t^2+t^{-2})L,
\end{eqnarray*}
where $L\cup O$ consits of the link $L$ to which  a trivial
knot component is added. 
The second family consists of elements of the form
\begin{eqnarray*}
L-tH-t^{-1}V,
  \end{eqnarray*}
where $L,H,V$ are links that are identical except inside an embedded ball
in which they look as depicted in Figure~\ref{crossing}, and additionally,
the crossing in this figure comes from {\em different} link components of
$L$.
The third family consists of elements of the form
\begin{eqnarray*}
L-\epsilon(tH-t^{-1}V),
\end{eqnarray*}
where again $L,H,V$ are links that are identical except inside an embedded ball
in which they look as depicted in Figure~\ref{crossing}, but this
time the crossing in the figure is the {\em self-crossing} of a link component of $L$ and $\epsilon$ is the sign of that crossing. These skein relations
have been derived in \cite{kirbymelvin} for the Reshetikhin-Turaev
theory.

Let $S^3$ be the 3-dimensional sphere. In $K_t(S^3)$ one has $$S_n(K)-\left<S_n(K)\right>\emptyset=0,$$ and in
$RT_t(S^3)$ one has $$S_n(K)-J(K,n)\emptyset=0;$$ moreover, as explained
in \cite{annrazvan}, in $RT_t(S^3)$ one has
$S_n(K)-(-1)^n\left<S_n(K)\right>\emptyset=0$.
This gives the skein theoretical explanation for the equality (\ref{jokau}).
In fact it is the cabling principle stated in \cite{kirbymelvin} that
implies that the $n$th colored Jones polynomial of a knot $K$ in $S^3$ is
the polynomial that is the coefficient of the empty link when
we write the skein $S_n(K)\in RT_t(M)$ in the basis consisting of the empty link. 

We extend the definition of the colored Jones polynomial to a knot $K$ in some
arbitrary 3-dimensional oriented manifold $M$ by stating that this polynomial
is the skein $S_n(K)$ in the skein module $RT_t(M)$. Of course if
$RT_t(M)$ is a free module and we know a basis, then we can express
$S_n(K)$ in that basis and we obtain a family of actual polynomials, but
this is less relevant for us at this moment, because we are actually
concerned with how $S_n(K),$ $n\in {\mathbb Z}$, relate to one another.

As a consequence of Theorem~2.1 in \cite{annrazvan} (see also \S 3.2 in the
same paper) the 
colored Jones polynomials of the curve $y$  satisfy 
\begin{eqnarray*}
t^{-2n-1}S_{p+n}(y)-t^{2n+1}S_{p-n-1}(y)=S_{2n}(x)(t^{-1}S_p(y)-tS_{p-1}(y)),
\end{eqnarray*}
(in the new context the factor $(-1)^n$ dissapears).
If we add the corresponding
relations for $n+1$ and $n-1$ and subtract the one for $n$ multiplied by $T_2(x)=x^2-2$ we obtain
the {\em homogeneous} recursive  relation
\begin{eqnarray*}
  &&t^{-2n-3}S_{p+n+1}(y)-t^{2n+3}S_{p-n-2}(y)-t^{-2n-1}(x^2-2)S_{p+n}(y)\\
  &&+t^{2n+1}(x^2-2)S_{p-n-1}(y)+t^{-2n+1}S_{p+n-1}(y)-
    t^{2n-1}S_{p-n}(y)=0.
\end{eqnarray*}
We  rewrite this as
\begin{eqnarray*}
  &&t^{-2n-3}S_{n+p+1}(y)+t^{2n+3}S_{n-p}(y)-t^{-2n-1}(x^2-2)S_{n+p}(y)\\
  &&-t^{2n+1}(x^2-2)S_{n-p-1}(y)+t^{-2n+1}S_{n+p-1}(y)
  +t^{2n-1}S_{n-p-2}(y)=0.
\end{eqnarray*}
We want to associate to this relation a polynomial following the guidelines
of \cite{garle}. Let us first construct
a coefficient ring for these polynomials.

If $K$ is a knot and $N(K)$ is an open regular neighborhood of $K$ in $S^3$,
then the module $RT_t(S^3\backslash N(K))$ has also a module structure over
the skein algebra of the boundary,  $RT_t(\partial N(K))$, defined by
gluing the cylinder over the boundary to the knot complement.
But  $\partial N(K)$ is the 2-dimensional torus ${\mathbb T}^2$,
and so $RT_t(S^3\backslash N(K))=RT_t({\mathbb T}^2)$. It has been shown
in \cite{geluri} that the latter is canonically isomorphic to $K_t({\mathbb T}^2)$ and the multiplicative structure of this algebra has been exhibited in
\cite{frohmangelca}. When identifying the boundary of the knot complement
with the standard torus, 
we require that the curve $(1,0)$ in ${\mathbb T}^2$ is identified with
the longitude of the knot $K$ while the curve $(0,1)$ is identified with
the meridian of the knot.
Note that $(0,1)\cdot \emptyset=x$, where $x$ is the curve
from Figure~\ref{basistorknot}.
If we consider the subalgebra of $RT_t({\mathbb T}^2)$ generated by $(0,1)$,
then this is the same as the polynomial ring $R={\mathbb C}[t,t^{-1}][x]$,
and we think of $x$ as both the variable of the polynomial ring and
as the curve $x$.
Now we view $RT_t(S^3\backslash N(K))$ as an $R$-module.

For a function 
\begin{eqnarray*}
f:{\mathbb Z}\rightarrow RT_t(S^3\backslash N(K)), 
\end{eqnarray*}
following \cite{garle} we define the operators
\begin{eqnarray*}
(Mf)(n)=t^{2n}f(n), \quad (Lf)(n)=f(n+1). 
\end{eqnarray*}
The operators $L$ and $M$ and their inverses generate  the ring ${\mathcal T}_x$
which is the quotient
\begin{eqnarray*}
R[L,L^{-1},M,M^{-1}]/(LM=t^2ML, LL^{-1}=L^{-1}L=1, MM^{-1}=M^{-1}M=1)
\end{eqnarray*} 
A {\em recurrence polynomial} of the function $f$ is an element $P\in {\mathcal T}_x$ such that
$Pf=0$.

Turning to our particular example we notice that
a recurrence polynomial in $R[L,L^{-1}, M,M^{-1}]$  for the function 
\begin{eqnarray*}
f:{\mathbb Z}\rightarrow S_n(y),
\end{eqnarray*}
whose values are the colored Jones polynomials of the curve $y$ in the complement of 
the $(2p+1,2)$ torus knot, is
\begin{eqnarray*}
  t^{-3}L^{p+1}M^{-1}+t^3L^{-p}M-t^{-1}(x^2-2)L^pM^{-1}-t(x^2-2)L^{-p-1}M\\+tL^{p-1}M^{-1}
  +t^{-1}L^{-p-2}M.
\end{eqnarray*}
We can then find a recurrence polynomial in the variables $L$ and $M$ only, namely 
\begin{eqnarray*}
 && t^{2p+5}L^{p+2}M( t^{-3}L^{p+1}M^{-1}+t^3L^{-p}M-t^{-1}(x^2-2)L^pM^{-1}\\&&-t(x^2-2)L^{-p-1}M
  +tL^{p-1}M^{-1}
  +t^{-1}L^{-p-2}M)
  \\&&=L^{2p+3}+t^{4p+8}L^{2}M^2-t^4(x^2-2)L^{2p+2}-t^{4p+8}(x^2-2)LM^2\\
  &&+
  t^8L^{2p+1}+t^{4p+8}M^2\\
  &&=L^{2p+3}-t^4(x^2-2)L^{2p+2}+t^8L^{2p+1}+t^{4p+8}[L^2-(x^2-2)L+1]M^2\\
  &&=[L^2-t^4(x^2-2)L+t^8]L^{2p+1}+t^{4p+8}[L^2-(x^2-2)L+1]M^2.
\end{eqnarray*}

\begin{remark}
  When $t=1$ (the ``classical'' setting of the Reshetikhin-Turaev theory
  as opposed to $t=-1$ for the Kauffman bracket) this polynomial factors
  as $$(L^2-(x^2-2)L+1)(L^{2p+1}+M^2).$$ 
  \end{remark}

\end{document}